\documentclass{amsart}

\usepackage{amsmath,amsthm,amssymb,amsfonts, amscd, bm, graphicx, mathrsfs}

\usepackage{youngtab}

\usepackage[cmtip,all]{xy}

\usepackage{mathrsfs}

\usepackage[pagebackref=true, colorlinks]{hyperref}

\hypersetup{colorlinks=true,citecolor=blue,linkcolor=blue,urlcolor=blue,
pdfstartview=FitH}

\newtheorem{thm}{Theorem}[section]
\newtheorem{lem}[thm]{Lemma}
\newtheorem{prop}[thm]{Proposition}

\theoremstyle{definition}

\theoremstyle{remark}

\theoremstyle{plain}

\numberwithin{equation}{section}

\makeatletter
\DeclareRobustCommand\widecheck[1]{{\mathpalette\@widecheck{#1}}}
\def\@widecheck#1#2{%
    \setbox\z@\hbox{\m@th$#1#2$}%
    \setbox\tw@\hbox{\m@th$#1%
       \widehat{%
          \vrule\@width\z@\@height\ht\z@
          \vrule\@height\z@\@width\wd\z@}$}%
    \dp\tw@-\ht\z@
    \@tempdima\ht\z@ \advance\@tempdima2\ht\tw@ \divide\@tempdima\thr@@
    \setbox\tw@\hbox{%
       \raise\@tempdima\hbox{\scalebox{1}[-1]{\lower\@tempdima\box
\tw@}}}%
    {\ooalign{\box\tw@ \cr \box\z@}}}
\makeatother

\begin{document}

\title{A conjecture for the regularized fourth moment of Eisenstein series}

%%====================================================================================================%%
%%====================================================================================================%%
%%====================================================================================================%%

\author{Goran Djankovi{\' c}}

\address{
University of Belgrade\\ 
Faculty of Mathematics\\ 
Studentski Trg 16, p.p. 550\\ 
11000 Belgrade, Serbia
}

\email{djankovic@matf.bg.ac.rs}

\author{Rizwanur Khan}

\address{
Texas A\&M University at Qatar\\ PO Box 23874\\ Doha, Qatar
}

\email{rizwanur.khan@qatar.tamu.edu}

%====================================================================================

\begin{abstract}
We formulate a version of the Random Wave Conjecture for the fourth moment of Eisenstein series which is based on Zagier's regularized inner product. We prove an asymptotic formula expressing the regularized fourth moment as a mean value of $L$-functions. This is an advantage over previous work in the literature, which has approached the fourth moment problem through truncated Eisenstein series and not yielded a suitable expression in terms of $L$-functions.
\end{abstract}

%==============================================================================



\thanks{The first author was  partially supported by  Ministry of Education, Science and Technological Development of  Republic of Serbia, Project no. 174008.}

\maketitle

%%====================================================================================================%%
%%====================================================================================================%%
%%====================================================================================================%%

\section{Introduction}

One of the main research themes in recent years in the theory of automorphic forms is the problem of mass distribution. Let $X=\Gamma \backslash \mathbb{H}$, where $\mathbb{H}$ is the upper half complex plane and $ \Gamma = SL_2(\mathbb{Z})$.  In his PhD thesis, Spinu \cite{Sp} obtained the following type of weak equidistribution result:
\begin{equation}
\int_{X} |{E}_A(z, \frac{1}{2} + iT)|^4 d\mu z \ll T^\epsilon,  
\end{equation}
where  $d \mu (z)=\frac{dx dy}{y^2}$ and ${E}_A(z,s)$ is the truncated Eisenstein series, which on the fundamental domain equals $E(z,s)$ for $\text{Im}(z)\le A$, and $E(z,s)$ minus its constant term for $\text{Im}(z)>A$. See the next section for a more careful definition. Spinu's result (see also \cite{Lu} for a closely related result) is in line with a much more general conjecture, called the Random Wave Conjecture. This conjecture was made for Eisenstein series in \cite[section 7.3]{HR}. In terms of moments this implies: for any even integer $p \ge 0$ and any nice compact $\Omega \subset X$, we should have
\begin{equation} \label{theconj}
\lim_{T \rightarrow \infty} \frac{1}{{\rm vol}(\Omega)}  \int_{\Omega} \Big|\frac{E_A(z, \frac{1}{2} + iT)}{\sqrt{2 \log T}}\Big|^p d\mu z = \frac{c_p}{{\rm vol}(X)^{p/2}},
\end{equation}
where $c_p$ is the $p$th moment of the normal distribution $\mathcal{N}(0, 1)$. The same conjecture is also made for $E(z,\frac{1}{2} + iT)$.  As we will see below, $\sqrt{2 \log T}$ roughly equals $\|E_A(\cdot, \frac{1}{2} + iT)\|_2$.

One would of course like to go beyond Spinu's upper bound and prove an asymptotic for the fourth moment of Eisenstein series. In \cite{BK}, this was achieved, conditional on the Generalized Lindel\"{o}f Hypothesis, for Hecke Maass forms of large eigenvalue when $\Omega=X$, and agreement was found with the RWC. Thus in analogy one would expect (\ref{theconj}) to also hold for $p=4$ and $\Omega=X$, and one may hope that the statement in this case can be proven unconditionally.  After all, such problems can be a bit easier for Eisenstein series -- for example, recall that the case $p=2$ of (\ref{theconj}) was first proven for Eisenstein series \cite{LS} before the analogue was proven for Hecke Maass forms \cite{Li, So}. 

What would the proof of such an asymptotic entail? The starting point in \cite{BK} is to relate the fourth moment of an $L^2$-normalized Hecke Maass form $f$ to $L$-functions. One uses the spectral decomposition and Plancherel's theorem to write
\begin{align}
\label{setup} \langle f^2, f^2 \rangle = \sum_{j\ge 1} |\langle f^2, u_j \rangle |^2 + \ldots,
\end{align} 
where the inner product is the Petersson inner product, $\{ u_j : j\ge 1\}$ is an orthonormal basis of Hecke Maass forms, and the ellipsis denotes the contribution of the Eisenstein spectrum and constant eigenfunction. Next one can use Watson's triple product formula to relate the squares of the inner products on the right hand side to central values of $L$-functions. Thus the problem is reduced to one of obtaining a mean value of $L$-functions. If one tries to mimic this set up for $E(z,\frac12+iT)$ in place of $f$, the first obvious difficulty encountered is that the left hand side of (\ref{setup}) does not even converge. To circumvent this, Spinu worked with the truncated $E_A(z,\frac12 +iT)$, which decays exponentially at the cusp. However a major drawback is that $E_A(z,\frac12 +iT)$ is not automorphic, so Spinu could not obtain a precise relationship with $L$-functions. He could only obtain an upper bound \cite[section 4.2]{Sp}.

The goal of this paper is to reformulate entirely the fourth moment problem for Eisenstein series. To make sense of $\langle E^2(\cdot,\frac12 + iT), E^2(\cdot,\frac12 + iT) \rangle$, we contend that it is more natural\footnote{We thank Matthew Young for suggesting this approach to us.} to use Zagier's regularized inner product \cite{Za}, which does converge. The basic idea of Zagier's method is that to kill off the growth of an automorphic form, one should not subtract off the constant term like Spinu does, but rather subtract off another Eisenstein series in such a way that the final object is square integrable and automorphic. This way we will end up with a precise relationship between a {\it regularized} fourth moment and $L$-functions. This is the first goal of our paper, and we will prove

\begin{thm} [Regularized fourth moment in terms of $L$-functions] \label{MainResult}  Let $\{ u_j : j\ge 1 \}$ denote an orthonormal basis of even and odd Hecke Maass cusp forms for $\Gamma$, ordered by Laplacian eigenvalue $\frac{1}{4}+it_j^2$, and let $\Lambda(s, u_j)$ denote the corresponding completed $L$-functions. Let $\xi(s)$ denote the completed Riemann $\zeta$ function. As $T \rightarrow \infty$, we have
$$
\int_{X}^{reg} |E(z, 1/2 +iT)|^4 d\mu(z)
$$
$$= \frac{24}{\pi} \log^2 T
+\sum_{\substack{j\ge 1}} \frac{\cosh(\pi t_j)}{2}  \frac{ |\Lambda(\frac{1}{2} + 2Ti, u_j)|^2   \Lambda^2(\frac{1}{2}, u_j) }{L(1, \text{sym}^2 u_j) \, |\xi(1+2Ti)|^4  } \;
 + O(\log^{5/3+\epsilon} T),
$$ 
for any $\epsilon>0$.
\end{thm}
\noindent This result is potentially very useful. We could try to obtain an asymptotic for the mean value of $L$-functions on the right hand side (and we will return to this problem in a future paper), thereby obtaining an asymptotic for the regularized fourth moment. This would be nice, but how would we know whether or not our answer is in agreement with the RWC? Thus the purpose of our second result is to translate the RWC to the setting of the regularized fourth moment. As defined in the next section, $\mathcal{D}_A$ is the part of the fundamental domain with $\text{Im}(z)\le A$.

\begin{thm} [RWC for the regularized fourth moment of Eisenstein series] \label{MainConj} Suppose that (\ref{theconj}) holds for $p=4$ and $\Omega=X$, and $p=4$ and $\Omega=\mathcal{D}_A$ for some $A=A(T)$ which tends to $\infty$ as $T \rightarrow \infty$. Then we have 
$$
\int_{X}^{reg} |E(z , \frac{1}{2} + iT)|^4 d \mu(z) \sim \frac{72}{\pi} \log^2 T.
$$
\end{thm}
\noindent We have already explained above why (\ref{theconj}) should be expected for $p=4$ and $\Omega=X$, even though for general $p$ we must restrict to compact sets. The other possibility $\Omega=\mathcal{D}_A$ is already included in the RWC when $A$ is fixed. But it is reasonable to conjecture that some effective error term will exist in (\ref{theconj}), so that taking $A$ which grows arbitrarily slowly should be permissible.

Both of our main results are based on careful calculations arising from the regularized inner product. The point is to offer a new viewpoint for the fourth moment and carefully put into place all leading constants, so that the relevant conjecture might be verified in the future using the theory of $L$-functions.

\section{Eisenstein series}

We recall the definition of Eisenstein series
$$
E(z, s) = \sum_{\gamma \in \Gamma_{\infty} \backslash \Gamma} {\rm Im}( \gamma z)^s =\frac{1}{2} y^s \sum_{\substack{ c, d \in \mathbb{Z} \\ (c, d)=1 }} \frac{1}{|c z +d|^{2s}},  \qquad z \in \mathbb{H},
$$
where  $\Gamma_{\infty}$ is the stabilizer of the cusp $\infty$ in $\Gamma$. The series is absolutely convergent in the half-plane ${\rm Re}(s)>1$ where it defines an automorphic function satisfying $\Delta E(z, s)=s(1-s) E(z, s)$, for the hyperbolic Laplacian $\Delta=-y^2(\frac{\partial^2}{\partial x^2} +\frac{\partial^2}{\partial y^2})$. 

\smallskip

The Eisenstein series can be meromorphically  continued to the whole $s$-plane and  $E(z, s)$  has the following Fourier expansion (for $s \neq 0, \frac{1}{2}, 1$)
$$
E\left(z, s \right)=y^{s} + \varphi(s) y^{1-s}  + \frac{2}{\xi(2s)}  \sum_{n \neq 0} \tau_{s-1/2}(|n|) \sqrt{ y}  K_{s-1/2}(2 \pi |n| y) e(n x).
$$
Here for complex $\alpha$, $\tau_{\alpha}(n)=\sum_{ab=n} (a/b)^{\alpha}$ is the generalized divisor sum and the scattering function $\varphi(s)$ can be explicitly expressed as
$$
\varphi(s)=\frac{\xi(2s-1)}{\xi(2s)}, \quad \text{where} \quad \xi(s)=\pi^{-s/2} \Gamma\left(\frac{s}{2} \right) \zeta(s).
$$
We will denote with $e(y, s):=y^{s} + \varphi(s) y^{1-s}$ the constant term of the Eisenstein series.

\smallskip

We denote with $\mathcal{D}=\{z \in \mathbb{H} \; | \; |z| \ge 1, |x| \le \frac{1}{2} \}$ the standard fundamental domain for $\Gamma \backslash \mathbb{H}$ and recall that its volume with respect to $d \mu$ is $\rm{vol}(X)= \rm{vol}(\mathcal{D})=\frac{\pi}{3}$. 

For a parameter $A >1$ we denote with $\mathcal{D}_A:=\{ z \in \mathcal{D} \; | \; {\rm Im}(z) \le A \}$ the corresponding truncated domain and with $\mathcal{C}_A= \mathcal{D} - \mathcal{D}_A$ the corresponding cuspidal region.

The \emph{truncated Eisenstein series} 
$$
E_A(z, s) = \left\{ \begin{array}{lr}
E(z, s), & z \in \mathcal{D}_A \\
E(z, s) -e(y, s), &  z \in \mathcal{C}_A
\end{array}
\right.
$$
is now rapidly decreasing in the cusp. Calculation of the $L^2$-norm of this truncated Eisenstein series is done in \cite{Sp}, Section 2.3, both in the case of the whole fundamental domain and in the case of the cuspidal region, as follows:
\begin{align}
\nonumber \int_{\mathcal{D}} |E_A(z, \frac{1}{2} + iT)|^2 d\mu z &=- \frac{\varphi'}{\varphi}(\frac{1}{2} +iT) + 2 \log A + \frac{A^{2 i T} \varphi(\frac{1}{2}  - iT) - A^{-2 i T} \varphi(\frac{1}{2}  + iT)}{2 i T}\\
\nonumber &=2 \log T + 2 \log A + O( (\log T)^{2/3 + \epsilon})\\
\label{2ndMoment_E_A_in_C_A} \int_{\mathcal{C}_A} |E_A(z, \frac{1}{2} + iT)|^2 d\mu z &= \frac{6}{A \pi} \log T + O(A^{-1}(\log T)^{2/3 + \epsilon}) +O(A^{-1}\log A), \quad T \rightarrow \infty.
\end{align}
Therefore in the compact truncated domain $\mathcal{D}_A$, since $E=E_A$ on $\mathcal{D}_A$, we have as $T \rightarrow \infty$
$$
\int_{\mathcal{D}_A} |E(z, \frac{1}{2} + iT)|^2 d\mu z \sim \left( \frac{\pi}{3} - \frac{1}{A}\right) \frac{6 \log T}{\pi} = {\rm vol}(\mathcal{D}_A) \frac{2 \log T}{{\rm vol}(X)},
$$
as long as $1<A\ll \log T$ say.
In other words, if we normalize the Eisenstein series as $$\tilde{E}(z, \frac{1}{2} + iT):=\frac{E(z, \frac{1}{2} + iT)}{\sqrt{2 \log T}},$$ we have
\begin{equation} \label{normalized2nd_D_A}
\lim_{T \rightarrow \infty } \frac{1}{{\rm vol}(\mathcal{D}_A)} \int_{\mathcal{D}_A} |\tilde{E}(z, \frac{1}{2} + iT)|^2 d\mu z= \frac{1}{{\rm vol}(X)}.
\end{equation}

If we denote $e^{i \theta(T)}:=\frac{\xi(1 + 2 i T)}{|\xi(1 + 2 i T)|}$, then the function $e^{i \theta(T)} \tilde{E}(z, \frac{1}{2} + iT)$ is real-valued, and the Random Wave Conjecture, as extended in \cite{HR}, predicts that $e^{i \theta(T)} \tilde{E}(z, \frac{1}{2} + iT)$ tends to Gaussian $\mathcal{N}(0, {\rm vol}(X)^{-1/2} )$ in distribution, when restricted to any compact and sufficiently regular subset $\Omega \subset X$. In particular, for the fourth moment ($c_4=3$), the conjecture predicts
$$
\lim_{T \rightarrow \infty } \frac{1}{{\rm vol}(\Omega)} \int_{\Omega} |\tilde{E}(z, \frac{1}{2} + iT)|^4 d\mu z=\frac{3}{{\rm vol}(X)^2}.
$$

\smallskip

By heuristic considerations and numerical experiments in \cite{HR}, the same limits should hold also for the normalized truncated Eisenstein series $$\tilde{E}_A(z, \frac{1}{2} + iT):=\frac{E_A(z, \frac{1}{2} + iT)}{\sqrt{2 \log T}}.$$ As explained, this should also include the case $\Omega=X$, in which case the conjecture is 
\begin{equation} \label{RWC_for_truncatedEisen}
\int_{X} |{E}_A(z, \frac{1}{2} + iT)|^4 d\mu z \sim \frac{36}{\pi} \log^2 T,   \qquad \text{as} \quad  T \rightarrow \infty.
\end{equation}

\section{Regularized inner product and regularized Plancherel formula}

We will make use of the regularization process given by Zagier in \cite{Za}. An adelic version with a representation theoretic interpretation and with an alternate way of defining regularization is recently given in \cite{MV}. 

Let $F(z)$ be a continuous $\Gamma$-invariant function on $\mathbb{H}$. It is called \emph{renormalizable} (in Zagier's terminology, or \emph{of controlled increase} in the terminology of \cite{MV}) if there is a function $\Phi(y)$ on $\mathbb{R}_{>0}$ of the form
\begin{equation} \label{Phi_def}
\Phi(y)=\sum_{j=1}^l \frac{c_j}{n_j!} y^{\alpha_j} \log^{n_j} y,
\end{equation}
with $c_j, \alpha_j \in \mathbb{C}$ and $n_j \in \mathbb{Z}_{\ge 0}$, such that
$$
F(z)= \Phi(y) + O(y^{-N})
$$
as $y \rightarrow \infty$, and for any $N>0$.

If $F(z)=\sum_{n= - \infty}^{\infty} a_n(y) e(n x)$ is the Fourier expansion of $F$ at the cusp $\infty$, in particular if $a_0(y)$ is its 0-term, and if no $\alpha_j$ equals 0 or 1, then the function
$$
R(F, s):=\int_0^{\infty} (a_0(y) - \Phi(y) )  y^{s-2} dy,
$$
where the defining integral converges for sufficiently large ${\rm Re}(s)$, can be meromorphically continued to all $s$ and has a simple pole at $s=1$. Then one can define the regularized integral with
\begin{equation} \label{reg_first_by_R(F,s)}
\int_{\Gamma \backslash \mathbb{H}}^{reg} F(z) d\mu(z) := \frac{\pi}{3}  {\rm Res}_{s=1} R(F, s).
\end{equation}
Moreover,  then the function $F(z)E(z, s)$ with $s\neq 0, 1$ is also renormalizable and in particular,  it can be shown that
$$
\int_{\Gamma \backslash \mathbb{H}}^{reg} F(z) E(z, s) d\mu(z) = R(F, s).
$$

It can be shown (see \cite{Za}) that the regularized integral can be written  also as
\begin{equation} \label{reg_int_area}
\int_{\Gamma \backslash \mathbb{H}}^{reg} F(z) d\mu(z)= \int_{\mathcal{D}_A} F(z)d \mu(z) +\int_{\mathcal{C}_A} (F(z) - \Phi(y))d \mu(z) - \hat{\Phi}(A),
\end{equation}
where the right-hand side is independent of the value of the parameter $A >1$ and $\hat{\Phi}(y)$ is in the case  $\alpha_j \neq 1$ for all $j$, given by the following explicit expression
$$
\hat{\Phi}(y)=\sum_{j=1}^l c_j \, \frac{y^{\alpha_j -1}}{\alpha_j -1} \, \sum_{m=0}^{n_j} \frac{ \log^m y}{m! (1- \alpha_j )^{n_j-m}}.
$$

Under the assumption that no $\alpha_j=1$, let $\mathcal{E}_{\Phi}(z)$ denote a linear combination of Eisenstein series $E(z, \alpha_j)$ (or suitable derivatives thereof) corresponding to all the exponents in (\ref{Phi_def}) with ${\rm Re}(\alpha_j) > 1/2$, i.e. such that $F(z) - \mathcal{E}_{\Phi}(z)=O(y^{1/2})$. Then the third, equivalent definition of regularization is given by
\begin{equation} \label{reg_subtract_Eisen}
\int_{\Gamma \backslash \mathbb{H}}^{reg} F(z) d\mu(z)=\int_{\Gamma \backslash \mathbb{H}} (F(z) -\mathcal{E}_{\Phi}(z)) d\mu(z).
\end{equation}

For example Zagier showed in \cite{Za} that for $s_1,s_2 \in \mathbb{C} \setminus \{ 0, 1\}$, $s_1 \neq s_2, 1-s_2$, we have 
\begin{equation} \label{2eisen_orthog}
\int_{\Gamma \backslash \mathbb{H}}^{reg} E(z, s_1) E(z, s_2)  d\mu(z)=0.
\end{equation}

On the other hand, for the regularized  product of the three Eisenstein series, Zagier (ibid. pg 431) obtained
\begin{equation} \label{3eisen}
\int_{\Gamma \backslash \mathbb{H}}^{reg} E(z, \frac{1}{2} + s_1) E(z, \frac{1}{2} + s_2) E(z, \frac{1}{2} + s_3)  d\mu(z)= 
\end{equation}
$$
=\frac{\xi(\frac{1}{2} +s_1 +s_2 +s_3) \xi(\frac{1}{2} +s_1 -s_2 +s_3) \xi(\frac{1}{2} +s_1 +s_2 -s_3) \xi(\frac{1}{2} +s_1 -s_2 -s_3)}{\xi(1+ 2 s_1) \xi(1+ 2 s_2) \xi(1+ 2 s_3)}.
$$
The right-hand side is of course symmetric in $s_1, s_2, s_3$ because of the functional equation $\xi(1-s)=\xi(s)$.
%%%%%%%%%%%%%%%%%%%%%%%%%%%%%%%%%%%%%%%%%%%%%%%%%%%%%%%%%%%%%%%%%%%%%%%%%%%

\smallskip

Since we are interested in the regularized product of 4 Eisenstein series, one can try to apply the definition (\ref{reg_first_by_R(F,s)}) directly. But already Zagier in \cite{Za}, pg. 431, discussed that in this case there is no useful closed-form expression for the result, as is  for the product of 3 Eisenstein series in (\ref{3eisen}). Therefore, we must proceed indirectly via a regularized Plancherel formula.

\smallskip

Now, let $G(z)$ be another renormalizable $\Gamma$-invariant function such that $G(z)=\Psi(y) + O(y^{-N})$ as $y \rightarrow \infty$ for any $N>0$, where $\Psi(y)=\sum_{k=1}^{l_1} \frac{d_k}{m_k!} y^{\beta_k} \log^{m_k}y$ with $d_k, \beta_k \in \mathbb{C}$. Then the product $F(z) \overline{G(z)}$ is also a renormalizable $\Gamma$-invariant function and
if $\alpha_j+ \overline{\beta_k} \neq 1$, for all $\alpha_j$ and $\beta_k$ appearing in $\Phi$ and $\Psi$ respectively,    the regularized inner product of  $F$ and $G$ can be defined as 
$$
\langle F, G \rangle_{reg} :=\int_{\Gamma \backslash \mathbb{H}}^{reg} F(z) \overline{G(z)} d\mu(z)=\int_{\Gamma \backslash \mathbb{H}} (F(z) \overline{G(z)} - \mathcal{E}_{\Phi \overline{\Psi}}(z)) d\mu(z).
$$
It is easy to see from (\ref{reg_int_area}) that this regularized product is a Hermitian form.

\smallskip

The regularized Plancherel formula from \cite{MV} is much more general, but for our purposes we will state and derive it entirely in classical situation of Zagier's paper \cite{Za}, much in the spirit of Lemma 4.1 from \cite{Yo}. 
Because of the cumbersome formulas, we will use the shorthand notation $E_s(z):=E(z, \frac{1}{2} +s)$, and   remind the reader not to confuse this with  the truncated Eisenstein series $E_A(z, s)$ which still have 2 arguments. 

\begin{prop} [\cite{MV}] \label{Prop_Regular_Planch} Let $F(z)$ and $G(z)$ be renormalizable functions on $\Gamma \backslash \mathbb{H}$ such that $F - \Phi$ and $G - \Psi$ are of rapid decay as $y \rightarrow \infty$, for some $\Phi(y)=\sum_{j=1}^l \frac{c_j}{n_j!} y^{\alpha_j} \log^{n_j} y$ and $\Psi(y)=\sum_{k=1}^{l_1} \frac{d_k}{m_k!} y^{\beta_k} \log^{m_k}y$. Moreover, let $\alpha_j \neq 1$, $\beta_k \neq 1$, ${\rm Re}(\alpha_j) \neq \frac{1}{2}$, ${\rm Re}(\beta_k) \neq \frac{1}{2}$, $\alpha_j + \overline{\beta_k} \neq 1$ and $\alpha_j \neq \overline{\beta_k}$, for all $j, k$. Then the following formula holds:
$$
\langle F(z), G(z) \rangle_{reg}=
$$
$$
=\langle F, \sqrt{3/ \pi}  \rangle_{reg} \langle  \sqrt{3/ \pi} , G  \rangle_{reg} +   \sum_j \langle F , u_j \rangle \langle  u_j, G \rangle  + \frac{1}{4 \pi} \int_{-\infty}^{\infty} \langle F, E_{it}\rangle_{reg}  \langle  E_{it}, G \rangle_{reg} dt
$$
$$
+ \langle F, \mathcal{E}_{\Psi}  \rangle_{reg}  +  \langle  \mathcal{E}_{\Phi}, G \rangle_{reg}.
$$

\end{prop}

\begin{proof} Because of the assumption ${\rm Re}(\alpha_j)\neq \frac{1}{2}$, ${\rm Re}(\beta_k)\neq \frac{1}{2}$, there exists some $\delta >0$ such that $F_1(z):= F(z) - \mathcal{E}_{\Phi}(z)= O(y^{1/2 - \delta})$ and $G_1(z):= G(z) - \mathcal{E}_{\Psi}(z)= O(y^{1/2 - \delta})$. For these $F_1(z)$ and $G_1(z)$ we have $F_1(z) \overline{G_1(z)} \in L^1(\Gamma \backslash \mathbb{H})$ and hence $\langle F_1, G_1 \rangle_{reg} = \langle F_1, G_1 \rangle$ (the usual Petersson inner product), while also  $F_1(z), G_1(z) \in L^2(\Gamma \backslash \mathbb{H})$ and hence one can apply the usual Plancherel formula for $\langle F_1(z), G_1(z) \rangle$, obtaining
$$
\langle F(z), G(z) \rangle_{reg}=\langle F_1(z) + \mathcal{E}_{\Phi}(z), G_1(z) +\mathcal{E}_{\Psi}(z) \rangle_{reg}=
$$
$$
=\langle F_1, \sqrt{3/ \pi}  \rangle \langle  \sqrt{3/ \pi} , G_1  \rangle +   \sum_j \langle F_1 , u_j \rangle \langle  u_j, G_1 \rangle  + \frac{1}{4 \pi} \int_{-\infty}^{\infty} \langle F_1, E_{it}\rangle  \langle  E_{it}, G_1\rangle dt
$$
$$
+ \langle F_1, \mathcal{E}_{\Psi}  \rangle_{reg}  +  \langle  \mathcal{E}_{\Phi}, G_1 \rangle_{reg}  + \langle \mathcal{E}_{\Phi}, \mathcal{E}_{\Psi} \rangle_{reg}.   
$$
Under our restrictions  on the parameters $\alpha_j, \beta_k$, all the inner products  on the  right hand side  are well-defined and moreover because of (\ref{2eisen_orthog}), we have that $\langle \mathcal{E}_{\Phi}, \mathcal{E}_{\Psi} \rangle_{reg}=0$, which also implies that $\langle F_1, \mathcal{E}_{\Psi}  \rangle_{reg}=\langle F, \mathcal{E}_{\Psi}  \rangle_{reg}$ and $\langle  \mathcal{E}_{\Phi}, G_1 \rangle_{reg}=\langle  \mathcal{E}_{\Phi}, G \rangle_{reg}$. Furthermore, for the products with cusp forms $u_j$ we have $\langle F_1 , u_j \rangle=\langle F  , u_j \rangle - \langle \mathcal{E}_{\Phi} , u_j \rangle= \langle F  , u_j \rangle$, the product with constant function is $\langle F_1  , 1 \rangle= \langle F - \mathcal{E}_{\Phi}  , 1 \rangle=\langle F  , 1 \rangle_{reg}$ by definition of regularization and $\langle F_1  , E_{it} \rangle=\langle F - \mathcal{E}_{\Phi}  , E_{it} \rangle_{reg}=\langle F  , E_{it} \rangle_{reg}$ since $\langle \mathcal{E}_{\Phi}, E_{it} \rangle_{reg}=0$, by (\ref{2eisen_orthog}). This finishes the proof.

\end{proof}

\section{Proof of Theorem \ref{MainResult}}

%%%%%%%%%%%%%%%%%%%%%%%%%%%%%%%%%%%%%
We want to apply this formula for the product of four Eisenstein series. By calculating the constant term of $F(z):=E(z, \frac{1}{2} + s_1) E(z, \frac{1}{2} + s_2)$, we find that
$$
\mathcal{E}_{\Phi}(z)=E(z, 1+ s_1+s_2) +c_1 E(z, 1-s_1 +s_2) + c_2 E(z, 1+s_1 -s_2) +c_1 c_2 E(z, 1-s_1 -s_2)
$$
where
$$
c_j=\varphi \left(\frac{1}{2} +s_j \right) =\frac{\xi(2 s_j)}{\xi(1+2s_j)},
$$
and we have the similar formula for $\mathcal{E}_{\Psi}(z)$ corresponding to $G(z):=E(z, \frac{1}{2} + s_3) E(z, \frac{1}{2} + s_4)$.
Hence, under the conditions on the parameters $s_j$, $1\le j \le 4$, described in Proposition \ref{Prop_Regular_Planch} ($\alpha_j=1 \pm s_1 \pm s_2, \; \beta_k=1 \pm s_3 \pm s_4$), we get
\begin{equation} \label{eq4reg}
\langle E_{s_1} E_{s_2}, \; E_{s_3} E_{s_4} \rangle_{reg}=\left\langle E(z, \frac{1}{2} + s_1) E(z, \frac{1}{2} + s_2), \; E(z, \frac{1}{2} + s_3) E(z, \frac{1}{2} + s_4) \right\rangle_{reg}=
\end{equation}
$$
=\frac{3}{\pi} \langle E_{ s_1} E_{s_2}, \; 1  \rangle_{reg} \langle  1 , \; E_{s_3} E_{s_4}  \rangle_{reg} 
+   \sum_{j \ge 1} \langle E_{ s_1} E_{s_2} , u_j \rangle \;  \langle  u_j, E_{ s_3} E_{s_4} \rangle  
$$
$$
+ \frac{1}{4 \pi} \int_{-\infty}^{\infty} \langle E_{ s_1} E_{s_2}, \; E_{it}\rangle_{reg} \; \langle  E_{it}, \; E_{ s_3} E_{s_4} \rangle_{reg} \, dt
$$
$$
+ \langle E_{ s_1} E_{s_2}, \; E_{\frac{1}{2} +s_3 +s_4} + c_3 E_{\frac{1}{2} -s_3 +s_4} +c_4 E_{\frac{1}{2} +s_3 -s_4} +c_3 c_4 E_{\frac{1}{2} -s_3 -s_4} \rangle_{reg}  $$
$$
+  \langle  E_{\frac{1}{2} +s_1 +s_2} + c_1 E_{\frac{1}{2} -s_1 +s_2} +c_2 E_{\frac{1}{2} +s_1 -s_2} +c_1 c_2 E_{\frac{1}{2} -s_1 -s_2}, \; E_{ s_3} E_{s_4} \rangle_{reg}. 
$$

%%%%%%%%%%%%%%%%%%%%%%%%%%%%%%%%%%%%%%%%%%%%%%%%%%%%%%%%%%%%%%%%%%%%%%%%%%%%%%%
\smallskip

For the cusp forms $u_j$, the triple products $\langle E_{ s_1} E_{s_2} , u_j \rangle$ can be evaluated by the standard unfolding argument (see Section 2 of \cite{LS}):

\begin{lem} Let $u_j(z)$ be a Hecke-Maass cusp form for the group $\Gamma$, that is an eigenvalue of the Laplace operator $\Delta u_j= (\frac{1}{4} + t_j^2) u_j$ and of all Hecke operators $T_n u_j=\lambda_j(n) u_j$, for all $n\ge 1$, which satisfies also $T_{-1} u_j(z)=\overline{u_j(-\overline{z})}=\epsilon_j u_j(z)$, $\epsilon_j=\pm 1$. Then it has the Fourier expansion $u_j(z)=\rho_j(1) \sum_{n \neq 0} \lambda_j(n) \sqrt{y} K_{it_j}(2 \pi |n| y)  e(nx)$ with $\lambda_j(-n)= \epsilon_j \lambda_j(n)$. Let $L(s, u_j)$ be the $L$-function associated to $u_j$, defined by analytic continuation from the Dirichlet series $\sum_{n \ge 1} \lambda_j(n)n^{-s}$. Then for $s_1, s_2 \neq \pm 1/2$, if $u_j$ is even ($\epsilon_j=1$) we have
\begin{equation} \label{tripleEEu}
\langle E(\cdot, 1/2 + s_1) E(\cdot, 1/2 +s_2) , u_j \rangle=\frac{\overline{\rho_j(1)}}{2} \frac{\Lambda(\frac{1}{2} + s_1 +s_2, u_j) \Lambda(\frac{1}{2} + s_1 - s_2, u_j)}{\xi(1+2s_1) \xi(1+2s_2)},
\end{equation}
where $\Lambda(s, u_j):=\pi^{-s} \Gamma(\frac{s+i t_j}{2}) \Gamma(\frac{s-i t_j}{2}) L(s, u_j)$ is the completed $L$-function corresponding to $u_j$. In the case of odd $u_j$ (i.e. $\epsilon_j=-1$), the triple product is $0$.
\end{lem}

{\em Remark:} The right hand side in (\ref{tripleEEu}) is symmetric in $s_1, s_2$, since for even $u_j$, we have the functional equation $\Lambda(s, u_j)=\Lambda(1-s, u_j)$. Moreover, we have the following formula relating the normalizing factor $\rho_j(1)$ with the symmetric square $L$-function:
$$
|\rho_j(1)|^2 = \frac{2 \cosh(\pi t_j)}{L(1, \text{sym}^2 u_j)}.
$$

\smallskip
%%%%%%%%%%%%%%%%%%%%%%%%%%%%%%%%%%%%%%%%%
Further,  for $s_1 \neq \pm s_2$ and $s_3 \neq \pm s_4$, by (\ref{2eisen_orthog}) the first term on the right-hand side in (\ref{eq4reg}) vanishes. Finally, using (\ref{3eisen}) for all regularized triple products of Eisenstein series, we arrive at
\begin{equation} \label{4Eisen_general_s_j}
\langle E_{s_1} E_{s_2}, \; E_{s_3} E_{s_4} \rangle_{reg}
\end{equation}
$$
=\sum_{\substack{j\ge 1 \\ \epsilon_j=1}} \frac{\cosh(\pi t_j)}{2}  \frac{\Lambda(\frac{1}{2} + s_1 +s_2, u_j) \Lambda(\frac{1}{2} + s_1 - s_2, u_j)  \Lambda(\frac{1}{2} + \overline{s}_3 +\overline{s}_4, u_j) \Lambda(\frac{1}{2} + \overline{s}_3 - \overline{s}_4, u_j)}{L(1, \text{sym}^2 u_j)\xi(1+2s_1) \xi(1+2s_2) \xi(1+2\overline{s}_3) \xi(1+2 \overline{s}_4)}
$$
$$
+\frac{1}{4 \pi} \int_{-\infty}^{\infty} \frac{\prod_{\delta_1, \delta_2 \in \{ \pm 1 \}}   \xi(\frac{1}{2} +t  i  +\delta_1 s_1 + \delta_2 s_2) \xi(\frac{1}{2} +t  i  +\delta_1 \overline{s}_3 + \delta_2 \overline{s}_4)}{|\xi(1+2ti)|^2  \xi(1+ 2s_1) \xi(1+ 2s_2)\xi(1+ 2\overline{s}_3)  \xi(1+ 2\overline{s}_4)}  dt
$$
$$
+\frac{\prod_{\delta_1, \delta_2 \in \{ \pm 1 \} }  \xi(1+\delta_1 s_1 +\delta_2 s_2 +\overline{s}_3 + \overline{s}_4 )}{\xi(1+2 s_1) \xi(1+2 s_2) \xi(2+2 \overline{s}_3 + 2\overline{s}_4)}  +\overline{c}_3 \frac{\prod_{\delta_1, \delta_2 \in \{ \pm 1 \} }  \xi(1+\delta_1 s_1 +\delta_2 s_2 -\overline{s}_3 + \overline{s}_4 )}{\xi(1+2 s_1) \xi(1+2 s_2) \xi(2-2 \overline{s}_3 + 2\overline{s}_4)}
$$
$$
+\overline{c}_4 \frac{\prod_{\delta_1, \delta_2 \in \{ \pm 1 \} }  \xi(1+\delta_1 s_1 +\delta_2 s_2 +\overline{s}_3 - \overline{s}_4 )}{\xi(1+2 s_1) \xi(1+2 s_2) \xi(2+2 \overline{s}_3 - 2\overline{s}_4)} +\overline{c}_3 \overline{c}_4 \frac{\prod_{\delta_1, \delta_2 \in \{ \pm 1 \} }  \xi(1+\delta_1 s_1 +\delta_2 s_2 -\overline{s}_3  - \overline{s}_4 )}{\xi(1+2 s_1) \xi(1+2 s_2) \xi(2-2 \overline{s}_3 - 2\overline{s}_4)}
$$
$$
+\frac{\prod_{\delta_1, \delta_2 \in \{ \pm 1 \} }  \xi(1+ s_1 + s_2 +\delta_1 \overline{s}_3 + \delta_2 \overline{s}_4 )}{\xi(2+2 s_1 + 2 s_2)   \xi(1+2 \overline{s}_3) \xi(1+2 \overline{s}_4) } + c_1 \frac{\prod_{\delta_1, \delta_2 \in \{ \pm 1 \} }  \xi(1- s_1 + s_2 +\delta_1 \overline{s}_3 + \delta_2 \overline{s}_4 )}{\xi(2 - 2 s_1 + 2 s_2)   \xi(1+2 \overline{s}_3) \xi(1+2 \overline{s}_4) }
$$
$$
+ c_2 \frac{\prod_{\delta_1, \delta_2 \in \{ \pm 1 \} }  \xi(1+ s_1 - s_2 +\delta_1 \overline{s}_3 + \delta_2 \overline{s}_4 )}{\xi(2+2 s_1 - 2 s_2)   \xi(1+2 \overline{s}_3) \xi(1+2 \overline{s}_4) }  + c_1 c_2 \frac{\prod_{\delta_1, \delta_2 \in \{ \pm 1 \} }  \xi(1 - s_1 - s_2 +\delta_1 \overline{s}_3 + \delta_2 \overline{s}_4 )}{\xi(2 - 2 s_1 - 2 s_2)   \xi(1+2 \overline{s}_3) \xi(1+2 \overline{s}_4) }.
$$

Let us denote the last eight terms (quotients of products of $\xi$-functions, coming from the regularization process) on the right hand side of (\ref{4Eisen_general_s_j}) with $\Xi_j$, $1\le j \le 8$, respectively with the order of appearance in (\ref{4Eisen_general_s_j}).
\smallskip

Now, let us choose for $s_j$ the following values: $s_1=iT$, $s_2=iT +\nu$, $s_3=iT$ and $s_4=iT+\overline{\eta}$, with complex parameters $\nu$ and $\eta$ satisfying $0 < {\rm Re}(\nu) < {\rm Re}(\eta) < \frac{1}{4}$. For these values  all the conditions from Proposition \ref{Prop_Regular_Planch} are satisfied. From (\ref{reg_subtract_Eisen}), we see that $\langle E_{iT} E_{iT +\nu}, \; E_{iT} E_{iT +\overline{\eta}} \rangle_{reg}$ is  continuous in $\nu, \eta$, and therefore, if we first let $\nu \rightarrow 0$ in (\ref{4Eisen_general_s_j}), keeping $\eta$ fixed, we get
\begin{equation} \label{4Eisen_eta}
\langle E_{iT}^2 , \; E_{iT} E_{iT +\overline{\eta}} \rangle_{reg}
\end{equation}
$$
=\sum_{\substack{j\ge 1 \\ \epsilon_j=1}} \frac{\cosh(\pi t_j)}{2}  \frac{\Lambda(\frac{1}{2} + 2Ti, u_j) \Lambda(\frac{1}{2}  -2Ti + \eta, u_j)  \Lambda(\frac{1}{2}, u_j) \Lambda(\frac{1}{2}  - \eta, u_j)}{L(1, \text{sym}^2 u_j) \, \xi^2(1+2Ti) \xi(1- 2Ti)  \xi(1- 2Ti +2 \eta)}
$$
$$
+\frac{1}{4 \pi} \int_{-\infty}^{\infty} \frac{ \xi^2(\frac{1}{2} +t  i)  \prod_{\pm}   \xi(\frac{1}{2} +t  i  \pm 2Ti) \xi(\frac{1}{2} +t  i  \pm 2Ti \mp \eta) \xi(\frac{1}{2} + ti \pm \eta)}{|\xi(1+2ti)|^2  \xi^2(1+ 2Ti) \xi(1- 2Ti)  \xi(1 -2Ti + 2 \eta)}  dt
$$  
$$
+ \sum_{j=1}^8 \Xi_j(T, \eta),
$$
where
\begin{align*}
\Xi_1(T, \eta) &= \frac{\xi(1+ \eta) \xi^2(1 -2Ti + \eta)  \xi(1- 4Ti +\eta)}{\xi^2(1+2 Ti)  \xi(2 -4Ti +2 \eta)} =: \xi(1+\eta)F_1(\eta),  \\
\Xi_2(T, \eta) &= \frac{  \xi^2(1+\eta ) \xi(1+2Ti +\eta) \xi(1-2Ti +\eta)}{\xi(1 -2Ti) \xi(1+2 Ti) \xi(2 + 2 \eta)} =: \xi^2(1+\eta ) F_2(\eta), \\
\Xi_3(T, \eta) &=  \frac{  \xi^2(1- \eta )  \xi(1+2Ti - \eta)  \xi(1+2Ti - 2\eta)  \xi(1-2Ti - \eta)}{\xi^2(1 + 2Ti) \xi(1 -2Ti + 2\eta) \xi(2- 2\eta)} =: \xi^2(1-\eta ) F_3(\eta),      \\
\Xi_4(T, \eta) &=  \frac{\xi(1-\eta) \xi^2(1 +2Ti - \eta) \xi(1 +2Ti -2 \eta) \xi(1 +4Ti - \eta) }{\xi(1+ 2Ti) \xi(1- 2Ti) \xi(1- 2Ti +2\eta) \xi(2 +4Ti -2\eta)} =: \xi(1-\eta ) F_4(\eta),   \\
\Xi_5(T, \eta) &=  \frac{\xi(1+\eta) \xi(1+ 2 Ti -\eta) \xi(1+2Ti +\eta) \xi(1+ 4Ti -\eta) }{\xi(2 +4Ti) \xi(1-2Ti) \xi(1-2Ti + 2 \eta)} =: \xi(1+\eta ) F_5(\eta),      \\
\Xi_6(T, \eta) &= \Xi_7(T, \eta)=   \frac{\xi(1- \eta) \xi(1+\eta) \xi(1+ 2Ti -\eta) \xi(1-2Ti + \eta)}{\xi(2) \xi(1+2Ti) \xi(1-2Ti +2 \eta)} =: \xi(1-\eta) \xi(1+\eta ) F_6(\eta),   \\
\Xi_8(T, \eta) &= \frac{\xi(1-\eta) \xi(1- 2Ti - \eta) \xi(1- 2Ti) \xi(1- 2Ti+ \eta) \xi(1- 4Ti + \eta)}{\xi^2(1 +2Ti) \xi(1- 2Ti +2 \eta)  \xi(2- 4Ti)} =: \xi(1-\eta ) F_8(\eta).
\end{align*}
 
Each of $\Xi_j$ has a pole at $\eta=0$, but the whole sum $\sum_{j=1}^8 \Xi_j$ has a removable singularity at $\eta=0$. This can be seen by grouping together $\Xi_1$ with $\Xi_8$, $\Xi_4$ with $\Xi_5$, and $\Xi_2 + \Xi_3$ with $2\Xi_6=\Xi_6 + \Xi_7$. More explicitly, if we denote with 
\begin{equation} \label{xi_Laurent}
\xi(s)=\frac{1}{s-1}+ a + b(s-1) + O((s-1)^2) 
\end{equation}
the Laurent expansion of $\xi(s)$ around $s=1$, we get the following expansions of $\Xi_j(\eta)$:
\begin{align*}
\Xi_1(\eta) &=     & \frac{F_1(0)}{\eta}   &+  F_1'(0) +a F_1(0)  &+ O(\eta), \\
\Xi_2(\eta) &=   \frac{F_2(0)}{\eta^2}  &+ \frac{F_2'(0)+ 2a F_2(0) }{\eta}   &+  (a^2 +2 b) F_2(0) +2 a F_2'(0) +\frac{1}{2} F_2''(0) &+ O(\eta), \\
\Xi_3(\eta) &=     \frac{F_3(0)}{\eta^2}  &+ \frac{F_3'(0)- 2a F_3(0) }{\eta}   &+  (a^2 +2 b) F_3(0) -2 a F_3'(0) +\frac{1}{2} F_3''(0) &+ O(\eta), \\
\Xi_4(\eta) &=     &  -\frac{F_4(0)}{\eta}   &  -F_4'(0) +a F_4(0)  &+ O(\eta), \\
\Xi_5(\eta) &=     &  \frac{F_5(0)}{\eta}   &+  F_5'(0) +a F_5(0)  &+ O(\eta), \\
\Xi_6(\eta) =\Xi_7(\eta) &=   - \frac{F_6(0)}{\eta^2}  &  - \frac{F_6'(0)}{\eta}    &+ (a^2 -2b) F_6(0) -\frac{1}{2} F_6''(0)   &+ O(\eta), \\
\Xi_8(\eta) &=     &   -\frac{F_8(0)}{\eta}   &  -F_8'(0) +a F_8(0)  &+ O(\eta).
\end{align*}

But, $F_2(0)=F_3(0)=F_6(0)=\frac{1}{\xi(2)}$, so the polar terms with $\frac{1}{\eta^2}$ cancel out in the sum. Further, $F_1(0)=F_8(0)$, $F_4(0)=F_5(0)$ and one calculates
$$
F_2'(0)=\frac{2}{\xi(2)} \left[ {\rm Re} \frac{\xi'}{\xi}(1+2Ti) -\frac{\xi'}{\xi}(2) \right], \qquad F_3'(0)=\frac{2}{\xi(2)} \left[ \frac{\xi'}{\xi}(2)   -3{\rm Re} \frac{\xi'}{\xi}(1+2Ti) \right]
$$
and
$$
F_6'(0)= -\frac{2}{\xi(2)} {\rm Re} \frac{\xi'}{\xi}(1+2Ti),
$$
from which it follows that $F_2'(0) + F_3'(0) -2 F_6'(0)=0$, a.e. the coefficient in front of $\frac{1}{\eta}$ also vanishes.   Therefore we can take $\eta \rightarrow 0$ in (\ref{4Eisen_eta}) and after  calculation of all the other required  derivatives appearing in
$$
\lim_{\eta \rightarrow 0} \sum_{j=1}^8 \Xi_j(T, \eta)=a (F_1(0) + F_4(0) + F_5(0) +F_8(0)) +  a^2 (F_2(0) + F_3(0) + 2F_6(0)) 
$$ 
$$
+F_1'(0) -F_4'(0) + F_5'(0) -F_8'(0) +2a (F_2'(0) - F_3'(0)) +\frac{1}{2}F_2''(0) +\frac{1}{2}F_3''(0)- F_6''(0),
$$ 
we obtain the following exact evaluation of the regularized fourth power of Eisenstein series:

\begin{prop} For any nonzero real $T$, we have:
\begin{equation} \label{reg4Eisen_exact}
\int_{\Gamma \backslash \mathbb{H}}^{reg} |E(z, 1/2 +iT)|^4 d\mu(z)=\langle E_{iT}^2 , \; E_{iT}^2  \rangle_{reg}
\end{equation}
$$
=\sum_{\substack{j\ge 1 \\ \epsilon_j=1}} \frac{\cosh(\pi t_j)}{2}  \frac{\Lambda(\frac{1}{2} + 2Ti, u_j) \Lambda(\frac{1}{2}  -2Ti , u_j)  \Lambda^2(\frac{1}{2}, u_j) }{L(1, \text{sym}^2 u_j) \, |\xi(1+2Ti)|^4  }
$$
$$
+\frac{1}{4 \pi} \int_{-\infty}^{\infty} \frac{ \xi^4(\frac{1}{2} +t  i)     \xi^2(\frac{1}{2} +t  i  + 2Ti) \xi^2(\frac{1}{2} +t  i  - 2Ti) }{|\xi(1+2ti)|^2  |\xi(1+ 2Ti)|^4   }  dt
$$  
$$
+ \frac{4}{\xi(2)} \left[  {\rm Re} \frac{\xi''}{\xi}(1+2Ti) +2 \left| \frac{\xi'}{\xi}(1+2Ti)  \right|^2  +{\rm Re} \frac{(\xi')^2}{\xi^2}(1+ 2Ti) \right.
$$
$$
\left. + 4(a - \frac{\xi'}{\xi}(2)) {\rm Re} \frac{\xi'}{\xi}(1+2Ti)  +2 \frac{(\xi')^2}{\xi^2}(2) - \frac{\xi''}{\xi}(2) - 2a \frac{\xi'}{\xi}(2) +a^2 \right]
$$
$$
+\frac{\xi^2(1+2Ti) \xi(1+4Ti)}{\xi^2(1-2Ti) \xi(2+4Ti)} \left[ 2a + 4 \frac{\xi'}{\xi}(1+2Ti) -2 \frac{\xi'}{\xi}(2+4Ti) \right]
$$
$$
+\frac{\xi^2(1-2Ti) \xi(1-4Ti)}{\xi^2(1+2Ti) \xi(2-4Ti)} \left[ 2a + 4 \frac{\xi'}{\xi}(1-2Ti) -2 \frac{\xi'}{\xi}(2-4Ti) \right],
$$
where $a=\lim_{s \rightarrow 1} (\xi(s) - (s-1)^{-1})$.
\end{prop}

\smallskip

{\em Remark:} The exact value of the constant $a$ is $\frac{C_0}{2} -\frac{\ln \pi}{2} -\ln 2 =-0.9769...$, where $C_0=0.57721...$ is Euler's constant. This is a consequence of the following two formulas: $\zeta(s)=\frac{1}{s-1} +C_0 + O((s-1))$ and $\frac{\Gamma'}{\Gamma}(1/2)= -C_0 -2 \ln 2$. Also we recall that $\xi(2)=\frac{\pi}{6}$.
%%%%%%%%%%%%%%%%%%%%%%%%%%%%%%%%%%%%%%%

\smallskip

Using Stirling's approximations $|\Gamma(\sigma +i t)|=e^{-\pi |t|/2}|t|^{\sigma-\frac{1}{2}} \sqrt{2 \pi} \{ 1+ O(|t|^{-1}) \}$ and $\frac{\Gamma'}{\Gamma}(s)= \log s + O(|s|^{-1})$ valid in a fixed vertical strip when $|t| \rightarrow \infty$, and classical estimates for the Riemann zeta-function on the edge of the critical strip \cite[section 6.3]{MoV},
\begin{equation} \label{zeta_estimates_Re=1}
(\log t)^{-2/3}(\log\log t)^{-1/3} \ll \zeta(s) \ll (\log t)^{2/3} \qquad \text{and} \qquad \frac{\zeta'}{\zeta}(s) \ll (\log t)^{2/3}(\log\log t)^{1/3}
\end{equation}
for
\begin{align*} 
s=\sigma+it, \ \ 1-\sigma \ll (\log t)^{-2/3}(\log\log t)^{-1/3}, 
\end{align*}
one obtains first that $\frac{\xi'}{\xi}(1 \pm 2Ti) \ll \log T$ and then that the contribution of the terms in the last two lines in the formula (\ref{reg4Eisen_exact}) is $O(\frac{\log^2 T}{T^{1/2}})$. Therefore the contribution on the right-hand side of (\ref{reg4Eisen_exact}) coming from the regularization process is
$$
\frac{24}{\pi} \left[  {\rm Re} \frac{\xi''}{\xi}(1+2Ti) +2 \left| \frac{\xi'}{\xi}(1+2Ti)  \right|^2  +{\rm Re} \frac{(\xi')^2}{\xi^2}(1+ 2Ti) \right] + O(\log T).
$$
Since $\frac{\xi'}{\xi}(s)=-\frac{\log \pi}{2} +\frac{1}{2} \frac{\Gamma'}{\Gamma} (\frac{s}{2}) +\frac{\zeta'}{\zeta}(s)$, from Stirling's approximation and (\ref{zeta_estimates_Re=1}) we obtain further that when $T \rightarrow \infty$
\begin{equation*}
2 \left| \frac{\xi'}{\xi}(1+2Ti)  \right|^2  +{\rm Re} \frac{(\xi')^2}{\xi^2}(1+ 2Ti) =\frac{3}{4} \log^2 T + O(\log^{5/3+\epsilon} T),
\end{equation*}
for any $\epsilon>0$.

\smallskip

\begin{lem} \label{lem_zeta_2ndDer} As $t \rightarrow \infty$, we have
$$
\frac{\zeta''}{\zeta}(1+ti) \ll \log^{4/3+\epsilon} t.
$$
\end{lem}

\begin{proof} We can use (\ref{zeta_estimates_Re=1}) and the Borel-Carath\'{e}odory lemma \cite[Lemma 6.2]{MoV} to get the bound
$$
\left( \frac{\zeta'}{\zeta} \right)'(1+ti) \ll \log^{4/3+\epsilon} t.
$$
This and (\ref{zeta_estimates_Re=1}) again imply the stated bound for $\frac{\zeta''}{\zeta}(1+ti)=(\frac{\zeta'}{\zeta})'(1+ti) + (\frac{\zeta'}{\zeta}(1+ti))^2$.
\end{proof}

\smallskip

From
$$
\frac{\xi''}{\xi}(s)=\frac{\log^2 \pi}{4} + \frac{1}{4} \frac{\Gamma''}{\Gamma}(\frac{s}{2}) +\frac{\zeta''}{\zeta}(s) -\frac{\log \pi}{2} \frac{\Gamma'}{\Gamma}(\frac{s}{2}) -(\log \pi)\frac{\zeta'}{\zeta}(s)+\frac{\Gamma'}{\Gamma}(\frac{s}{2})\frac{\zeta'}{\zeta}(s),
$$
using (\ref{zeta_estimates_Re=1}), Lemma \ref{lem_zeta_2ndDer} and another well-known approximation $\frac{\Gamma''}{\Gamma}(\frac{1}{2} +Ti)=\log^2 T + O(\log T)$, we obtain also the asymptotic
\begin{equation*}
{\rm Re} \frac{\xi''}{\xi}(1+2Ti) = \frac{1}{4} \log^2 T  + O(\log^{5/3+\epsilon} T).
\end{equation*}

\smallskip

On the other hand, the contribution of the continuous spectrum in (\ref{reg4Eisen_exact}) i.e.  the integral on the right-hand side is of a  smaller  size, being bounded by the integral in the following Lemma:

\begin{lem} \label{Lem_bound_cont_spec} For $T \ge 1$ we have:
$$
\int_{-\infty}^{\infty} \frac{ |\xi(\frac{1}{2} +t  i) |^4    |\xi(\frac{1}{2} +t  i  + 2Ti)|^2 |\xi(\frac{1}{2} +t  i  - 2Ti)|^2 }{|\xi(1+2ti)|^2  |\xi(1+ 2Ti)|^4   }  dt  \ll T^{-1/6},
$$
for some absolute implicit constant. 

\end{lem}

\begin{proof} This is exactly Proposition 3.4 in \cite{Sp}. For completeness we briefly repeat here the argument. After employing Stirling's asymptotic formula for Gamma functions and  after splitting the integral $\int_{-\infty}^{+\infty}=2 \int_0^{3T} + 2 \int_{3T}^{+\infty}$, one can see easily that the contribution in the range $\int_{3T}^{+\infty}$ decays exponentially with $T$. Therefore, one needs to bound the integral
$$
\int_0^{3T} \frac{|\zeta(\frac{1}{2} + ti)|^4}{(1+ |t|)} \cdot \frac{|\zeta(\frac{1}{2} +(t+2T)i)|^2}{(1+|t+2T|)^{1/2}} \cdot \frac{|\zeta(\frac{1}{2} +(t-2T)i)|^2}{(1+|t-2T|)^{1/2}} \cdot \frac{e^{\frac{\pi}{2} (4T-|t-2T| -|t +2T|)   }}{|\zeta(1+2ti)|^2 |\zeta(1+2Ti)|^4} dt.
$$
By (\ref{zeta_estimates_Re=1}) the fourth ratio can be bounded by $T^{\varepsilon}$, the third ratio can be bounded by convexity bound, while the second ratio can be bounded using the subconvexity bound $\zeta(\frac{1}{2} +ti) \ll (1+|t|)^{\theta + \varepsilon}$ for some $\theta <\frac{1}{6}$, which is available and sufficient. The bound follows by the fourth moment estimate $\int_0^{3T} \frac{|\zeta(\frac{1}{2} + ti)|^4}{1+ |t|} dt \ll T^{\varepsilon}$, for any $\varepsilon >0$.
\end{proof}

Therefore after putting together everything in this section, we obtain the asymptotic formula in Theorem \ref{MainResult}. Note that the theorem has dropped the condition $\epsilon_j=1$. This is fine because when $\epsilon_j=-1$, we have $\Lambda(\frac{1}{2}, u_j)=0$ and the summand vanishes.

\smallskip

\section{Proof of Theorem \ref{MainConj}} \label{lastSection}

The regularized fourth moment of Eisenstein series $E(z, 1/2 + iT)$ 

$$
 \langle E_{iT}^2, \; E_{iT}^2\rangle_{reg}=\int_{\Gamma \backslash \mathbb{H}}^{reg} |E(z , \frac{1}{2} + iT)|^4 d \mu(z)
$$
can also be  expressed directly using (\ref{reg_int_area}). The corresponding function $\Phi(y)$ is given by
$$
\Phi(y)=|e(y, 1/2 + i T)|^4= c^2 y^{2-4Ti} + 4c y^{2 -2Ti}  +6 y^2 +4 \overline{c} y^{2+ 2Ti}  + \overline{c}^2 y^{2+ 4Ti}
$$
where $c=\varphi(1/2 + iT)=\frac{\xi(1-2Ti)}{\xi(1+ 2Ti)}$, so in particular $|c|=1$. Hence we get

$$
\int_{\Gamma \backslash \mathbb{H}}^{reg} |E(z , \frac{1}{2} + iT)|^4 d \mu(z)=
$$
$$
=\int_{\mathcal{D}_A} |E(z, 1/2 +iT)|^4  d\mu(z) + \int_{\mathcal{C}_A} (|E(z, 1/2 +iT)|^4  - |e(y, 1/2 + i T)|^4) d\mu(z)  -\hat{\Phi}(A),
$$
with $\hat{\Phi}(A)$ given explicitly by  
$$
\hat{\Phi}(A)= c^2 \frac{A^{1-4Ti}}{1-4Ti} + 4c \frac{A^{1-2Ti}}{1-2Ti}  +6A +4 \overline{c} \frac{A^{1+2Ti}}{1+2Ti} + \overline{c}^2 \frac{A^{1+4Ti}}{1+4Ti}.
$$
In particular, $\hat{\Phi}(A) \ll A$. This will be an admissible error for all the values of the truncation parameter  in the range $1 < A  \ll \log T$.

\smallskip

Hence the difference between the regularized integral of $|E|^4$ and the integral $\int_X |E_A|^4 d \mu$  with the truncated Eisenstein series considered in \cite{Sp} is

$$
\int_{\Gamma \backslash \mathbb{H}}^{reg} |E(z , \frac{1}{2} + iT)|^4 d \mu(z)
- \| E_A(\cdot , \frac{1}{2} + iT) \|^4_4
$$
$$
=\int_{\mathcal{D}_A} |E|^4 d \mu  + \int_{\mathcal{C}_A} ( |E_A +e|^4 - |e|^4) d \mu -\hat{\Phi}(A) - \int_{\mathcal{D}_A} |E|^4 d \mu - \int_{\mathcal{C}_A} |E_A|^4 d \mu
$$
\begin{equation} \label{formula_difference}
=   \int_{\mathcal{C}_A} ( e^2 \overline{E}_A^2 + \overline{e}^2 E_A^2 + 4 |e|^2 |E_A|^2 ) d \mu  +2 \int_{\mathcal{C}_A} ( |E_A|^2 \overline{E}_Ae + |E_A|^2 E_A \overline{e} ) d \mu    -\hat{\Phi}(A),
\end{equation}
since $\int_{\mathcal{C}_A} E_A \overline{e} |e|^2 d \mu=\int_{\mathcal{C}_A} \overline{E}_A e |e|^2 d \mu=0$. 

\smallskip

Here, the first integral in the cuspidal region can be explicitly computed. From the integral representation 
$$
K_{\nu}(y)=\int_0^{+\infty}e^{-y \cosh t} \cosh(\nu t) dt, \qquad {\rm Re}(\nu) > - \frac{1}{2},
$$
we see that $K_{iT}(y)$ is real for $y >0$, $T \in \mathbb{R}$ and hence for $z \in \mathcal{C}_A$
\begin{equation} \label{eq_Fourier_real}
\xi(1+ 2Ti) E_A(z, 1/2 +iT)= 4 \sum_{n=1}^{\infty} \tau_{iT}(n) \, y^{\frac{1}{2}} K_{iT}(2 \pi n y) \, \cos (2 \pi n x)
\end{equation}
is also real-valued.  Using this and the functional equation for $\xi(s)$, after a short calculation one gets that the first integral in (\ref{formula_difference}) is equal to
$$
\int_{\mathcal{C}_A} \left( \frac{12 y}{|\xi(1+ 2Ti)|^2} + \frac{6 y^{1+2Ti}}{\xi^2(1 - 2Ti)}  + \frac{6 y^{1-2Ti}}{\xi^2(1 + 2Ti)}  \right)  \xi^2(1+ 2Ti) E_A^2(z, 1/2 +iT) d \mu(z).
$$

\smallskip

Therefore we need to calculate  the twisted integrals of the second moment of the truncated Eisenstein series in the cuspidal region 
$$
I_{\eta}:= \int_{\mathcal{C}_A} y^{1+ \eta} \xi^2(1+ 2Ti) E_A^2(z, 1/2 +iT) d \mu(z),
$$
for the values of parameter $\eta \in \{ 0, \pm 2Ti \}$. Substituting here the Fourier expansion (\ref{eq_Fourier_real}) we obtain
$$
I_{\eta}= 16  \int_A^{+\infty} \int_0^1 y^{1+ \eta} \left( \sum_{n=1}^{\infty} \tau_{iT}(n) \, y^{\frac{1}{2}} K_{iT}(2 \pi n y) \, \cos (2 \pi n x) \right)^2 \frac{dx dy}{y^2}
$$
$$
=8 \sum_{n=1}^{\infty}  \tau_{iT}^2(n)  \int_A^{\infty}   K_{iT}^2(2 \pi n y) y^{\eta} d y =8 \sum_{n=1}^{\infty}  \tau_{iT}^2(n) (2 \pi n)^{-1 - \eta} g(2 \pi A n ),
$$
where
$$
g(x):=\int_x^{\infty} K_{iT}^2 (y) y^{\eta} d y.
$$
The Mellin transform of this function is equal to
$$
G(s):=\int_0^{\infty} g(x) x^s \frac{dx}{x} =\frac{1}{s} \int_0^{\infty} K_{iT}^2(x) x^{\eta + s} dx
$$
$$
=\frac{2^{\eta -2 + s}}{s \Gamma(1 + \eta +s)} \Gamma^2\left(\frac{1 +\eta +s}{2} \right) \Gamma \left( \frac{1 +\eta +s}{2}  +iT \right) \Gamma \left( \frac{1 +\eta +s}{2}  -iT \right)
$$
by integration by parts and the Mellin-Barnes formula  \cite{GR}, 6.576.4
$$
\int_0^{\infty} K_{\mu}(y)K_{\nu}(y) y^s \frac{dy}{y}=\frac{2^{s-3}}{\Gamma(s)} \prod_{\pm, \pm} \Gamma\left(\frac{s\pm \mu \pm \nu}{2} \right).
$$
By the inverse Mellin transform we have $g(x)=\frac{1}{2 \pi i} \int_{(3)} G(s) x^{-s} ds$ (where the integration is over the line ${\rm Re}(s)=3$) and so we get
$$
I_{\eta}=\frac{8}{(2 \pi)^{1+ \eta}} \frac{1}{2 \pi i} \int_{(3)} G(s) (2 \pi  A)^{-s} \sum_{n=1}^{\infty} \frac{\tau_{iT}^2(n)}{n^{s+ 1 + \eta}} ds.
$$
Here, since $\tau_{iT}(n)=\sigma_{2iT}(n) n^{-iT}$, we have by Ramanujan's identity 
$$
\sum_{n=1}^{\infty} \frac{\tau_{iT}^2(n)}{n^{s+ 1 + \eta}}= \frac{\zeta^2(s +1 + \eta) \zeta(s+1+\eta +2 Ti) \zeta(s+1+\eta -2 Ti)}{\zeta(2 s +2+ 2 \eta)},
$$
which then  gives

$$
I_{\eta}= \frac{1}{2 \pi i} \int_{(4)} \frac{A^{1-s}}{s-1} \frac{\xi^2(s  +\eta) \xi(s  +\eta + 2Ti) \xi(s  +\eta -2Ti)}{\xi(2s  + 2 \eta)} ds.
$$
The integrand is rapidly decreasing in vertical strips and it is regular on the line ${\rm Re}(s)=\frac{1}{2}$ (for all three values of the parameter $\eta$), so we can shift the line of integration from ${\rm Re}(s)=4$ to ${\rm Re}(s)=\frac{1}{2}$:
\begin{equation} \label{shifting_contour}
I_{\eta}=\mathcal{R}_{\eta} + \frac{1}{2 \pi i} \int_{(1/2)} \frac{A^{1-s}}{s-1} \frac{\xi^2(s  +\eta) \xi(s  +\eta + 2Ti) \xi(s  +\eta -2Ti)}{\xi(2s  + 2 \eta)} ds,
\end{equation}
where $\mathcal{R}_{\eta}=\sum_P R_{\eta, P}$ is the sum of residues $R_{\eta, P}$ of the poles $P$ that we encounter.

\smallskip

In the case $\eta=0$, the integrand  has two simple poles at $s=1 \pm 2Ti$ with the residues
$$
R_{0, 1-2Ti}=-\frac{A^{2Ti}  \xi^2(1-2Ti) \xi(1-4Ti)}{2Ti \, \xi(2-4Ti)} \; \; \text{and} \; \; R_{0, 1+2Ti}=\frac{A^{-2 Ti} \xi^2(1+2Ti) \xi(1+4Ti)}{2Ti \, \xi(2+4Ti)}
$$
and the triple pole at $s=1$ with residue
$$
R_{0, 1}=\frac{|\xi(1+2Ti)|^2 }{\xi(2)} \cdot \left[  \left|\frac{\xi'}{\xi}(1+2Ti) \right|^2  +  {\rm Re} \frac{\xi''}{\xi}(1+2Ti)    +\frac{1}{2} \log^2 A  \right.
$$
$$
\left. -2 (\log A) {\rm Re} \frac{\xi'}{\xi}(1+2Ti)  +2  \left( \frac{\xi'}{\xi}(2) -a \right) \left(\log A   -2 {\rm Re} \frac{\xi'}{\xi}(1+2Ti) \right)  \right.
$$
$$
\left.    +4 \left( \frac{\xi'}{\xi}(2) \right)^2 -4 a \frac{\xi'}{\xi}(2) +a^2 +2 b -2 \frac{\xi''}{\xi}(2) \right],
$$
where the constants $a$ and $b$ are as in (\ref{xi_Laurent}).

\smallskip

In the case $\eta=2Ti$, the integrand  has the simple pole at $s=1-4Ti$ with the residue
$$
R_{2Ti, 1-4Ti}= - \frac{A^{4Ti} \xi^2(1-2Ti) \xi(1-4Ti)}{4Ti \, \xi(2 -4Ti)}
$$
and two double poles at $s=1$ and $s= 1-2Ti$ with the corresponding residues

%%%%%%%%%%%%%%%%%%%%%%%%%%%%%%%%%%%%%%%%%%%%%%%%%%%%%%%%%%%%%%%%%%%%%%%%%%%%%%%%%%%%%
$$
R_{2Ti, 1}=\frac{\xi^2(1+2Ti) \xi(1+4Ti)}{\xi(2+4Ti)} \left[ a -\log A +\frac{\xi'}{\xi}(1+4Ti) +2 \frac{\xi'}{\xi}(1+2Ti) -2 \frac{\xi'}{\xi}(2+ 4Ti) \right]
$$
%%%%%%%%%%%%%%%%%%%%%%%%%%%%%%%%%%%%%%%%%%%%%%
and
$$
R_{2Ti, 1-2Ti}=-  \frac{A^{2Ti} \xi(1+2Ti) \xi(1-2Ti)}{2Ti \, \xi(2)} \left[  2 {\rm Re}\frac{\xi'}{\xi}(1+2Ti)-\log A  +\frac{1}{2Ti} +2 a -2 \frac{\xi'}{\xi}(2)\right].
$$

\smallskip

In the case $\eta=-2Ti$, the integrand  has the simple pole at $s=1+4Ti$ with the residue
$$
R_{-2Ti, 1+4Ti}=\frac{A^{-4Ti} \xi^2(1+2Ti) \xi(1+4Ti)}{4Ti \, \xi(2+ 4Ti)}
$$
and two double poles at $s=1$ and $s= 1+2Ti$ with the corresponding residues
%%%%%%%%%%%%%%%%%%%%%%%%%%%%%%%%%%%%%%%%%%%%%%%%%%%%%%%%%%%%%%%%

$$
R_{-2Ti, 1}=  \frac{\xi^2(1-2Ti) \xi(1-4Ti)}{\xi(2-4Ti)} \left[ a  - \log A + \frac{\xi'}{\xi}(1-4Ti) +2\frac{\xi'}{\xi}(1-2 Ti) -2\frac{\xi'}{\xi}(2-4Ti) \right]
$$
and
$$
R_{-2Ti, 1+2Ti}=\frac{A^{-2Ti} \xi(1+2Ti) \xi(1-2Ti)}{2 T i \, \xi(2)} \left[  2 {\rm Re}\frac{\xi'}{\xi}(1+2Ti) -\log A  -\frac{1}{2Ti} + 2a -2 \frac{\xi'}{\xi}(2) \right].
$$

In particular, we have 
$$ 
\frac{1}{\xi^2(1 \pm 2Ti)}\mathcal{R}_{\mp 2Ti} \ll (\log T)^2 T^{-1/2},
$$ 
when $T \rightarrow \infty$.
%%%%%%%%%%%%%%%%%%%%%%%%%%%%%%%%%%%%%%%%%%%%%

\smallskip

The contribution of the integrals on the shifted line in (\ref{shifting_contour}) is bounded in the following Lemma:

\begin{lem} \label{Lemma_bound_shifted_integrals} For any $\eta \in \{ 0, \pm 2Ti \}$ with $T>1$, we have
$$
 \int_{-\infty}^{\infty} \left| \frac{A^{\frac{1}{2} -ti}}{\frac{1}{2} -ti} \cdot \frac{\xi^2(\frac{1}{2} +ti  +\eta) \xi(\frac{1}{2} +ti  +\eta + 2Ti) \xi(\frac{1}{2} +ti  +\eta -2Ti)}{ \xi^2(1+ 2Ti) \xi(1+ 2ti  + 2 \eta)} \right| dt \ll A^{1/2} T^{-1/6},
$$
with an absolute implicit constant.

\end{lem}

\begin{proof} The analysis is similar to that in Lemma \ref{Lem_bound_cont_spec}. The case $\eta=0$ was treated in \cite{Sp}, section 4.3.2, where the bound $O(A^{1/2} T^{-1/6})$ is obtained. Here, we treat the case $\eta=2Ti$ (for $\eta=-2Ti$, the value of the integral is the same). Using Stirling's formula, we see that the integrand is bounded by
$$
A^{1/2} \frac{e^{\frac{\pi}{4} (4T -|t|  -|t +4T|) }}{(1+|t|)^{5/4} (1+|t+ 2T|)^{1/2} (1+|t+ 4T|)^{1/4}}
$$
$$
\times  \left| \frac{\zeta(\frac{1}{2} +ti) \; \zeta^2(\frac{1}{2} +(t +2 T)i) \; \zeta(\frac{1}{2} +(t +4T)i  )}{ \zeta^2(1+ 2Ti) \; \zeta(1+ (2t+4T)i )} \right|.
$$
Using subconvexity estimate $\zeta(1/2 + ti) \ll (1+|t|)^{\theta + \epsilon}$, for all $\epsilon >0$ and some $\theta < \frac{1}{6}$ for the zeta-functions in the numerator and (\ref{zeta_estimates_Re=1}) for the zeta-functions in the denominator, this is further bounded by
$$
A^{1/2} T^{\epsilon} \frac{e^{\frac{\pi}{4} (4T -|t|  -|t +4T|) }}{(1+|t|)^{\frac{5}{4}  -\theta -\epsilon} (1+|t+ 2T|)^{\frac{1}{2} -2\theta -\epsilon} (1+|t+ 4T|)^{\frac{1}{4} -\theta - \epsilon}}.
$$
We split the integration into 3 ranges: $\int_{-\infty}^{\infty}= \int_{-\infty}^{-4T} + \int_{-4T}^{0} + \int_{0}^{\infty}$. In the first and the third range we have an exponential decay of the integrand and so we have that in these ranges the integrals are bounded respectively by
$$
\int_{0}^{\infty}  \ll A^{1/2} T^{\epsilon} \int_{0}^{\infty} \frac{e^{- \frac{\pi}{2} t  }}{(1+|t|)^{\frac{5}{4}  -\theta -\epsilon}  T^{\frac{1}{2} -2\theta + \frac{1}{4} -\theta}} dt \ll A^{1/2} T^{-1/4}
$$
and
$$
\int_{-\infty}^{-4T} \ll  A^{1/2} T^{\epsilon} \int_{-\infty}^{-4T} \frac{e^{\frac{\pi}{2}(t +4T)  }}{T^{\frac{5}{4}  -\theta +\frac{1}{2} -2\theta} (1+|t+ 4T|)^{\frac{1}{4} -\theta - \epsilon}} dt \ll A^{1/2} T^{-5/4}.
$$
In the middle range the integral is bounded by
$$
\ll  A^{1/2} T^{\epsilon} \int_{-4T}^0 \frac{ dt }{(1+|t|)^{\frac{5}{4}  -\theta } (1+|t+ 2T|)^{\frac{1}{2} -2\theta } (1+|t+ 4T|)^{\frac{1}{4} -\theta }}
$$
$$
\ll  A^{1/2} T^{ \theta - \frac{1}{4} +\epsilon} \int_{-2T}^0 \frac{ dt }{(1+|t|)^{\frac{5}{4}  -\theta } (1+|t+ 2T|)^{\frac{1}{2} -2\theta } } \ll A^{1/2} T^{ 3\theta - \frac{3}{4} +\epsilon} \ll A^{1/2} T^{ -1/4 }.
$$

Therefore, in the cases $\eta=\pm 2Ti$, we get an even better bound $O(A^{1/2} T^{ -1/4 })$.
\end{proof}

\vspace{1cm}

After we collect everything together, and use asymptotic formulas for $\frac{\xi'}{\xi}(1+2iT)$ and $\frac{\xi''}{\xi}(1+2iT)$ already seen in the previous section, we get that the  contribution of the first integral in (\ref{formula_difference}) is
$$
\frac{12}{|\xi(1+ 2Ti)|^2} I_0 + \frac{6}{\xi^2(1-2Ti)} I_{2Ti}   + \frac{6}{\xi^2(1+2Ti)} I_{-2Ti}
$$
$$
=\frac{12}{|\xi(1+ 2Ti)|^2} {\mathcal{R}}_0 + \frac{6}{\xi^2(1-2Ti)} \mathcal{R}_{2Ti}   + \frac{6}{\xi^2(1+2Ti)} \mathcal{R}_{-2Ti}  + O( A^{1/2} T^{-1/6} )
$$
\begin{equation} \label{eq_calculation_E^2e^2}
= \frac{36}{\pi} \log^2 T + O(\log^{\frac53+\epsilon} T), 
\end{equation}
for the range $1 < A \ll \log T$. The main contribution is coming from $R_{0,1}$.
 
\smallskip

Putting together all our calculations in this section, we end up with the following Proposition:

\begin{prop} \label{prop_difference} When $T \rightarrow \infty$, for any value of the truncation parameter $1 < A \ll \log T$  we have 
$$
\int_{X}^{reg} |E(z , \frac{1}{2} + iT)|^4 d \mu(z)=
$$
\begin{equation} \label{eq_final_difference_asymptotic}
\int_{X} |E_A(z, \frac{1}{2} +iT)|^4 d \mu(z)
 +\frac{36}{\pi} \log^2 T   + 2 \int_{\mathcal{C}_A} ( |E_A|^2 \overline{E}_Ae + |E_A|^2 E_A \overline{e} ) d \mu \;
 + O(\log^{\frac53+\epsilon} T).
\end{equation}

\end{prop}
\noindent The first integral on the right hand side of (\ref{eq_final_difference_asymptotic}) is asymptotic to $\frac{36}{\pi} \log^2 T$ under assumption (\ref{RWC_for_truncatedEisen}). The integral over cuspidal region $\mathcal{C}_A$ in  (\ref{eq_final_difference_asymptotic}) is bounded by
\begin{equation} \label{Cauchy-Schwarz}
4 \int_{\mathcal{C}_A} \left|E_A^3(z, \frac{1}{2} +Ti) e(y, \frac{1}{2} +Ti) \right|   d \mu(z) \le 4 \left( \int_{\mathcal{C}_A} |E_A|^4 d\mu \right)^{1/2} \left( \int_{\mathcal{C}_A} |e E_A|^2 d\mu \right)^{1/2}.
\end{equation} 
The second integral on the right hand side of (\ref{Cauchy-Schwarz}) is $\sim \frac{6}{\pi} \log^2 T$; this is implicit in the calculation of the first integral in (\ref{formula_difference}). Under the RWC, the first integral on the right hand side of (\ref{Cauchy-Schwarz}) can be bounded by
$$
\int_{X} |E_A|^4 d\mu - \int_{\mathcal{D}_A} |E_A|^4 d\mu \ll ({\rm vol}(X)-{\rm vol}(\mathcal{D}_A)) \log^2 T \ll A^{-1} \log^2 T,
$$
which is $o(\log^2 T)$ if $A$ grows arbitrary slowly to infinity as $T\to \infty$. This way the right hand side of (\ref{eq_final_difference_asymptotic}) is asymptotic to $\frac{72}{\pi} \log^2 T $.

\bigskip

{\bf Acknowledgment.} Work on this project started when the first author visited Texas A\&M University at Qatar. He wishes to thank that institution for hospitality and  excellent working conditions.

%====================================================================================================================================================================

\end{document}